\documentclass[12pt]{article}

\usepackage{amsmath,amssymb,amscd,amsthm,amsfonts}
\usepackage{graphicx}
\usepackage{hyperref}
\usepackage{dsfont}

\newtheorem{theorem}{Theorem}[section]
\newtheorem*{theoremp}{Theorem}
\newtheorem{lemma}[theorem]{Lemma}
\newtheorem{claim}[theorem]{Claim}
\newtheorem{corollary}[theorem]{Corollary}

\newtheorem*{problem}{Problem}

\newtheorem{question}[theorem]{Question}
\numberwithin{figure}{section}

\makeatletter
\newtheorem*{rep@theorem}{\rep@title}
\newcommand{\newreptheorem}[2]{
\newenvironment{rep#1}[1]{
 \def\rep@title{#2 \ref{##1}}
 \begin{rep@theorem}}
 {\end{rep@theorem}}}
\makeatother

\newreptheorem{theorem}{Theorem}

\newcommand{\C}{\mathcal{C}}
\newcommand{\Z}{\mathbb{Z}}
\newcommand{\rr}{\mathds{R}}

\title{Measure Partitions Using Hyperplanes with Fixed Directions}

\author{Roman N. Karasev\thanks{Supported by the Dynasty foundation.} \and Edgardo Rold\'an-Pensado \and Pablo Sober\'on }

\begin{document}

\maketitle

\begin{abstract}
We study nested partitions of $R^d$ obtained by successive cuts using hyperplanes with fixed directions.  We establish the number of measures that can be split evenly simultaneously by taking a partition of this kind and then distributing the parts among $k$ sets.  This generalises classical necklace splitting results and their more recent high-dimensional versions.  With similar methods we show that in the plane, for any $t$ measures there is a path formed only by horizontal and vertical segments using at most $t-1$ turns that splits them by half simultaneously, and optimal mass-partitioning results for chessboard colourings of $R^d$ using hyperplanes with fixed directions.
\end{abstract}

\section{Introduction}

Mass partitioning results are a cornerstone of the interaction between topology and discrete geometry.  Besides being interesting geometric problems in their own right, they are highly motivated by their applications.  We mention a few key examples.

In discrete geometry, the recent polynomial partitioning-type results of Guth and Katz were used as their first step to nearly solve Erd\H{o}s' distinct distance problem \cite{Guth:tu} (see \cite{Sheffer:2014vy} for the state of the art on the subject). In combinatorics, one striking example is the solution to the necklace splitting problem \cite{Alon:1987ur, Goldberg:1985jr} which will be described precisely later on.  Mass partitions are also motivated by their applications to geometric range queries \cite{Matousek:1994kq}.

The quintessential measure partitioning result is the Ham Sandwich theorem.

\begin{theoremp}[Ham Sandwich]
Given $d$ probability measures $\mu_1, \mu_2, \dots, \mu_d$ in $\rr^d$, there is a hyperplane which simultaneously splits all of them evenly.
\end{theoremp}

There are some additional conditions imposed on the measures, namely that they vanish on every hyperplane.  For simplicity, in what follows we only consider measures which are absolutely continuous with respect to the Lebesgue measure of the space in question.  When the space of partitions is compact, approximation arguments can usually extend the results to larger families of functions.  However, one has to take additional steps to ensure that no new problems arise in doing so.

In general, the type of results which we are interested in are solutions to instances of the following problem:

\begin{problem}
Let $\mathcal{H}$ be a family of partitions of $\rr^d$ and $\mathcal{F}$ be a set of families of measures.  Under what conditions can we guarantee that for every $F \in \mathcal{F}$ there is a partition in $\mathcal{H}$ which splits $F$ in a prescribed manner?
\end{problem}

This can be thought as a way of testing the complexity of $\mathcal{H}$.  For instance, which families $\mathcal{H}$ satisfy that they can simultaneously split any $k$ measures in $\rr^d$ evenly?

In this paper, we study families of partitions formed by successive hyperplane cuts, where the directions hyperplanes and other parameters are fixed in advance.  With cuts by one hyperplane moving freely in $\rr^d$, the ham sandwich shows that any $d$ measures can be split.  However, if the direction of the hyperplane is fixed, only one measure can be guaranteed to be split.  Thus, the complexity of the partition must be increased in some other way.

In Section \ref{section-neckalces}, we describe a class of families of nested partitions by hyperplane cuts with prescribed directions that can split large numbers of measures at the same time.  The results in that section are closely related to the classic necklace splitting problem.

\begin{theoremp}[Necklace theorem, Hobby and Rice 1965 \cite{Hobby:1965bh}]
Given $t$ probability measures $\mu_1, \mu_2, \dots, \mu_t$ in $\rr$, there is a partition of $\rr$ using $t$ points such that the resulting parts may be distributed among two sets $A$ and $B$ with
\[
\mu_j (A) = \mu_j (B) = \frac12
\]
for all $j$.
\end{theoremp}

Notice that if the measures are not absolutely continuous, but instead each is distributed evenly among an even number of points, the problem is completely combinatorial.  The discrete version was first proved by Goldberg and West \cite{Goldberg:1985jr} using topological methods, and a different proof was obtained by Alon and West \cite{Alon:1985cy}. The setting of the problem is usually described as a problem of two thieves who stole an (open) necklace with pearls of $t$ types and wish to split it evenly (i.e. each gets the same number of each kind of pearl) among themselves by cutting the necklace and distributing the resulting pieces.  However, they want to use the minimum number of cuts possible.  This was later improved by Alon by solving the same problem with $k$ thieves \cite{Alon:1987ur}.

\begin{theoremp}[Alon 1987]
Given $t$ probability measures $\mu_1, \mu_2, \ldots, \mu_t$ in $\rr$, there is a partition of $\rr$ using $t(k-1)$ points such that the resulting parts may be distributed among $k$ sets $A_1, A_2, \ldots, A_k$ so that
\[
\mu_j (A_i) = \frac1k
\]
for all $i,j$.
\end{theoremp}

In both cases, the number of points used is optimal.  In Section \ref{section-neckalces} we generalise this result to $\rr^d$, where the cuts are made with hyperplanes with fixed directions.  This further generalises a $d$-dimensional version by \v{Z}ivaljevi\'c and de Longueville of the theorem above \cite{deLongueville:2006uo}, where the partition was made using hyperplanes with fixed directions, but only choosing out of $d$ possible directions.  In the case when $k$ is a prime power, we also optimise on the number of parts in the resulting partition, which has not been done before.

When dealing with a fair distribution problem, there are two aspects to take care in the partition.  One is the topological complexity of the space of partitions involved, and the other is the combinatorial complexity.  The methods by de Longueville and \v{Z}ivaljevi\'c optimised the topological complexity, as the number of hyperplane cuts they used was the best possible.  In our partitions, we optimise also the combinatorial complexity.  Given the hyperplanes that make up the partition, we only allow a small number of ways to distribute the parts.  Moreover, as the space of partitions considered is slightly different, we use simpler methods than the ones needed in \cite{deLongueville:2006uo}.

There is yet another variation of the Necklace splitting theorem by \v{Z}ivaljevic \cite{Zivaljevic:2013vla} which involves partitions of simplices induced by conical $(d-1)$-dimensional faces induced by a point in the interior of the simplex and some $(d-2)$-dimensional polytopal sphere in the boundary of the simplex.  In Section \ref{section-karasev} we show a unified topological approach that generalises the results in Section \ref{seccion-ajedrez} and \cite{Zivaljevic:2013vla}, thus containing or implying all known high-dimensional versions of the necklace splitting theorem.

In the partitions mentioned, we have no control over how the distribution is made.  However, when $k = 2$, given a partition of $\rr^d$ induced by hyperplanes there is a natural way to split the sections of the partition.  Namely, we colour the space like a chessboard, where every two parts that share a boundary of dimension $(d-1)$ go to different sets.  For this to be possible, we do not consider nested partitions, but allow all hyperplanes to extend completely.  Similar results appear in \cite{Alon:1985cy}.  In Section \ref{seccion-ajedrez} we describe mass partitioning results of this kind where we give more freedom to the hyperplanes determining the partition.

Another example of a mass partition  result is the polynomial version of the ham sandwich theorem by Stone and Tukey \cite{Stone:1942hu}.  In the case when $d=2$, it says that polynomials of degree $k-1$ are enough to split evenly any $k$ measures.
\begin{theoremp}[Stone, Tukey 1942]
Given $t$ probability measures $\mu_1, \mu_2, \ldots, \mu_t$ in $\rr^2$, there are is a polynomial $p$ of degree at most $t-1$ so that 
\[
\mu_i ( \{(x,y): y > p(x)\} ) = \frac12
\]
for all $i$.
\end{theoremp}

In Section \ref{section-paths}, we give a ``fixed directions'' version of this result.  This gives a positive answer to a conjecture by Mikio Kano, communicated to the authors by Ruy Fabila-Monroy.  Namely, we prove that just like polynomials of degree at most $t-1$, paths along two directions with at most $t-1$ turns can split any $t$ measures.  We present this result before the high-dimensional ones because the parametrisation of the space of partitions is much simpler and motivates the constructions used later.

\begin{reptheorem}{teorema-caminons-introduccion}
Let $t$ be a positive integer.  For any $t$ measures in $\rr^2$, there is a path formed by only horizontal and vertical segments with at most $t-1$ turns that splits $\rr^2$ into two sets of equal size in each measure.
\end{reptheorem}

Moreover, the path that we obtain is $y$-monotone.  One downside from our construction is that the path may go ``through infinity'' in the horizontal direction several times.  The case when $t=2$ was solved previously in~\cite{UNO:2009wk}.

It should be noted that the number of turns is optimal.  In order to see this, it suffices to concentrate each measure $\mu_i$ near a point $x_i$, so that no two points $x_i, x_j$ share a coordinate.  The path must essentially go through each point, and between consecutive points it needs at least one turn.  The proof of this result is contained in Section \ref{section-paths}.  It should be noted that there are similar results by Sergey Bereg \cite{Bereg:2009ch}, where it is shown that any two measures in the plane may be equipartitioned into $t$ parts simultaneously, where the boundaries for the sections of the equipartition are the union of at most $t-1$ vertical or horizontal segments.  In Bereg's results, the sections do not need to be convex.

Finally, in Section \ref{section-negative} we mention some examples of partitions similar to the ones we use that cannot split as many measures as their degrees of freedom.  This shows that the precise structure of our results cannot be removed altogether.

As mentioned at the beginning of the paper, the methods used are inescapably topological.  The most common proof scheme of this kind of results is what is known as the \emph{test map scheme}.  The objective is to parametrise the space of partitions by a space $X$, and map it by a function $f$ to a space $Y$ related to the way we are splitting the measures.  Ideally, the symmetries of the problem should induce actions of a group $G$ on both spaces so that $f$ is equivariant (i.e. $f(gx) = g \cdot f(x)$).  Then, the problem of the existence of certain partitions is related to the problem of the non-existence of certain equivariant maps.

For instance, let $t, k$ be positive integers, $[k] = \{1,2, \ldots, k\}$ and $n= t(k-1)+1$.  Let $\Sigma_k$ the symmetric group of permutations of $k$ elements.  The topological $n$-fold join 

\[
EG^{n-1}:= \underbrace{[k] * [k] * \cdots * [k]}_n
\]
 is an $(n-2)$-connected topological space with a natural action of $\Sigma_k$.  Namely, for a permutation $\sigma \in \Sigma_k$ and $p = \alpha_1 m_1 + \alpha_2 m_2 + \ldots + \alpha_n m_n \in EG^{n-1}$, we define $\sigma (p) = \alpha_1 \sigma (m_1) + \alpha_2 \sigma (m_2) + \ldots + \alpha_n \sigma (m_n)$.

Also, notice that the space $\rr^{k-1} = \{ (x_1, x_2, \ldots, x_k): x_1 + x_2 + \ldots + x_k = 0\}$ has an action of $\Sigma_k$, which simply permutes its coordinates.  This induces an action on $\rr^{n-1} = \rr^{t(k-1)} = \rr^{k-1} \oplus \rr^{k-1} \oplus \ldots \oplus \rr^{r-1}$.

The main topological tool needed for sections \ref{section-neckalces} and \ref{section-karasev} is the following theorem.

\begin{theoremp}[\"Ozaydin, 1987 \cite{Oza87}]
Let $k,t$ be positive integers and $n=t(k-1)+1$.  If $k$ is a prime power, then for every $\Sigma_k$-equivariant map
\[
f: EG^{n-1} \to \rr^{n-1}
\]
there is an $x_0$ such that $f(x_0) = 0$.  Moreover, if $k$ is not a prime power, there are maps $f$ as above whose image does not contain $0$.
\end{theoremp}

For the proof of this theorem, if $k=p^m$, then it is sufficient to consider the $\left( \mathbb{Z}_p \right)^m$-action in these spaces to show the existence of a zero.  It should be noted that even though \cite{Oza87} is an unpublished preprint, the results have been used and reproved several times, for instance in \cite{Volovikov:1996up}.

In the case when $k$ is a prime number instead of a prime power, a simpler proof of Theorem \ref{teorema-dimension-grande} can be obtain by using the following theorem by Dold instead of the result by \"Ozaydin.

\begin{theoremp}[Dold's theorem \cite{Dold:1983wr}]
Let $G$ be a finite group, $|G|>1$, $X$ be an $n$-connected space with a free action of $G$, and $Y$ be a (paracompact) topological space of dimension at most $n$ with a free action of $G$. Then there is no $G$-equivariant map $f:X \rightarrow Y$.
\end{theoremp}

One should note that when $X$ is an $n$-dimensional sphere and $Y$ is an $(n-1)$-dimensional sphere, both equipped with their natural $\Z_2$ action, the theorem above is simply the Borsuk--Ulam theorem.  For a description of these topological tools along their applications to discrete geometry and combinatorics, we recommend \cite{matousek2003using}.

\section{Paths using two directions}\label{section-paths}

In this section we show that a certain family of partitions of $\rr^2$ can be used to split evenly and simultaneously several measures.  Some of the resulting partitions are shown to come from $y$-monotone paths that use few number of turns.

Let $M = (m_1, m_2, \ldots, m_n)$ be a $(0,1)$-vector of length $n$.  We define $\mathcal{C}_M$ as the space of partitions of $\rr^2$ into two sets $A$ and $B$ which can be obtained in the following way.
\begin{itemize}
    \item First, choose numbers $-\infty=y_0\le y_1\le \dots \le y_{n-1}\le y_n=\infty$ and divide $\rr^2$ into the strips $S_i=\{(x,y):y_{i-1}\le y \le y_{i}\}$ for $i=1,2\dots,n$.
    \item If $m_i=0$, choose whether $S_i\subset A$ or $S_i\subset B$.
    \item If $m_i=1$, choose $-\infty\le x_i\le\infty$ and divide $S_i$ into $S_i^L=\{(x,y)\in S_i:x\le x_i\}$ and $S_i^R=\{(x,y)\in S_i:x\ge x_i\}$.  Then choose whether $S_i^L\subset A$ and $S_i^R\subset B$ or $S_i^L\subset B$ and $S_i^R\subset A$.
\end{itemize}

There is a natural action of $\Z_2$ acting on $\mathcal{C}_M$ given by $-(A,B) = (B,A)$. Note that when $y_i=y_{i+1}$, the corresponding strip has zero width. If additionally $y_i=\pm\infty$, then the strip is empty.

We give $\mathcal{C}_M$ the topology induced by the possible values of the variables $y_i$ and $x_i$. However, we consider each $x_i$ to be taken in $S^1$ so that the case $x_i=-\infty$ with $S_i^R\subset A$ (resp. $S_i^L\subset B$) is the same as the case $x_i=\infty$ with $S_i^L\subset A$ (resp. $S_i^R\subset B$). Furthermore, we identify any two partitions where the only difference is the choice of $x_i$ in the strips $S_i$ with empty interior.

Given a $(0,1)$-vector $M$, we denote by $w(M)$ the number of symbols $1$ in it.  Also, we denote by $S^n$ the boundary of the $(n+1)$-dimensional octahedron in $\rr^{n+1}$, which is homeomorphic to a $n$-dimensional sphere.  Namely, let
\[
S^k = \{ (x_1, x_2, \ldots, x_{k+1}) : |x_1|+|x_2|+\ldots + |x_{k+1}| = 1\},
\]
together with the usual action of $\Z_2$ on it.

\begin{lemma}\label{lem:topo}
Let $M$ be a $(0,1)$-vector of length $n$, and let $t= n+ w(M)$. Then there is a $\Z_2$-equivariant homeomorphism between $\mathcal{C}_M$ and $S^{t-1}$.
\end{lemma}

\begin{proof}
We prove this lemma by induction on the length of $M$.  Clearly, $\mathcal{C}_{( 0)} \cong S^0$ and $\mathcal{C}_{(1) }\cong S^1$ by homoemorphisms preserving the $\Z_2$ action.  Given two vectors $M$ and $N$, we denote by $M * N$ the vector formed by $M$ followed by $N$.

We will show that the space $\mathcal{C}_{M*N}$ is homeomorphic to the topological join $\mathcal{C}_M * \mathcal{C}_N$.  In order to simplify the proof, we can parametrize the extended real numbers $[-\infty, \infty]$ by the interval $[0,1]$, so that we can consider $\mathcal{C}_M$ as a space of partitions of the square $[0,1] \times [0,1]$ instead of $\rr^2$.

Given, $x \in \mathcal{C}_M$, $y \in \mathcal{C}_n$ and $s \in [0,1]$, we assign a partition of $\C_{M*N}$ to the element $sx + (1-s)y \in \C_M * \C_N$.  In order to do this, simply shrink the partition $x$ vertically by a factor of $s$ and use it as a partition of the rectangle $ [0,1]\times [1-s,1]$.  Then shrink the partition $y$ by a factor of $1-s$ and use it as a partition of the rectangle $ [0,1] \times [0,1-s] $.  This mapping is clearly an equivariant homeomorphism between the two spaces.

It is known that for any non-negative integers $l,m$ we have $S^l * S^m \cong S^{l+m+1}$, completing the proof.
\end{proof}

\begin{theorem}\label{theorem-partitiones-planas}
Let $M$ be a $(0,1)$-vector of length $n$, and $t= n+w(m)$.  For any $t-1$ probability measures $\mu_1, \mu_2, \ldots, \mu_{k-1}$ in $\rr^2$ which vanish on any line, there exists $(A,B) \in \C_M$ such that
\[
\mu_i (A) = \mu_i (B) = 1/2
\]
for all $i$.
\end{theorem}

\begin{proof}
Consider the function
\begin{align*}
f: \C_M &\rightarrow \rr^{t-1} \\
(A,B) &\mapsto \left(\mu_1 (A) - \frac12, \mu_2 (A) -\frac12, \ldots, \mu_{t-1} (A) - \frac12 \right)
\end{align*}
It is clear that $f(A,B) = -f(B,A)$. Since $\C_M \cong S^{t-1}$, by the Borsuk--Ulam theorem there is a partition $(A,B)\in \C_M$ such that $f(A,B) = 0$, as desired.
\end{proof}

In order to tackle the problem with paths, we are interested in $(0,1)$-vectors $M$ such that partitions in $\C_M$ can also be obtained by using a path with few turns.
We denote by $1_t$ the vector formed of exactly $t$ entries $1$, and by $1_t * 0$ the vector formed by $t$ entries $1$ followed by a single entry $0$.  We say that a path is \emph{stair-like} if it is $y$-monotone, uses only vertical and horizontal segments and is not self-intersecting. Notice that we allow the path to go through infinity in the horizontal direction.  We also allow some of the vertical segments to be at infinity.

\begin{figure}[htc]
\centering
\includegraphics{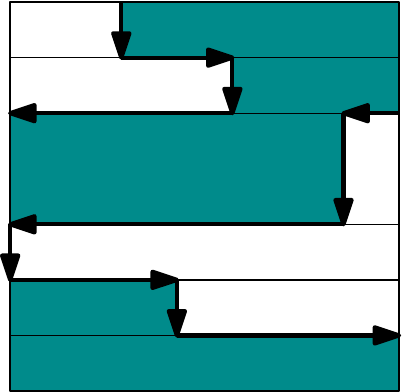}
\caption{Example of a partition induced by $(1,1,1,1,1,0)$ and its corresponding path.}
\label{fig-camino}
\end{figure}

The following lemma is not difficult and is left to the reader (see Figure \ref{fig-camino}).

\begin{lemma}\label{lemma-caminos}
Any partition induced by the vector $1_t$ can be also obtained by a stair-like path using at most $2t-2$ turns, and any partition induced by the vector $1_t*0$ can be induced by a stair-like path using at most $2t-1$ turns.
\end{lemma}

Combining Theorem \ref{theorem-partitiones-planas} and Lemma \ref{lemma-caminos}, we immediately obtain the result mentioned in the introduction.

\begin{theorem}\label{teorema-caminons-introduccion}
Given $k$ probability measures $\mu_1, \mu_2, \ldots, \mu_t$, there is a stair-like path using at most $t-1$ turns which splits $\rr^2$ into two sets $A,B$ so that
\[
\mu_i (A) = \mu_i (B) = \frac12
\]
for all $i$.
\end{theorem}

Going through infinity is necessary if the path is required to be $y$-monotone. It is easy to construct four measures concentrated around points that require this. We say that a path is \emph{bounded} if it does not go through infinity.

\begin{question}
Given any $t$ measures $\mu_1, \mu_2, \ldots, \mu_k$ in the plane, is there a bounded path formed by only vertical and horizontal segments, using at most $t-1$ turns that splits each measure by half?  If so, can the path be required to be non-self-intersecting?
\end{question}

\section{Partitions in high dimensions}\label{section-neckalces}

The partitions used in the previous section can be generalised inductively to higher dimensions.  In this section we describe families of partitions formed by nested hyperplane cuts which can split simultaneously large numbers of partitions.

It should be noted that partitions which are the result of nested power diagrams have provided generalisations of the ham sandwich theorem, dividing measures evenly into more than one part \cite{Karasev:2014gi, Soberon:2012kp}.  However, in this case we wish to restrict the directions of the cutting hyperplanes.

For this purpose we construct, recursively on $n$, sets $V^n_d$ which dictate how we are allowed to partition $\rr^d$.

\begin{itemize}
    \item The set $V^0_d$ is simply $\{ \emptyset \}$, where $\emptyset$ is the empty set.
    \item For $n >0$, $V^n_d$ is the set of triples $(v, Y_1, Y_2)$ where $v$ is a direction in $\rr^d$, and there is a non-negative $j \le n-1$ such that $Y_1 \in V^j_d$ and $Y_2 \in V^{n-1-j}_d$.
\end{itemize}

To each $Y \in V^n_d$, we assign a space $\mathcal{C}^n_Y$ of partitions of $\rr^d$ with some additional structure ($\mathcal{C}^n_Y$ is in fact an ordered multiset).  For $n=0$, $\C^n_{\emptyset}$ contains only the trivial partition of $\rr^d$ into one set.  For $n > 0$, and $Y = (v, Y_1, Y_2)$, the set $\C^n_Y$ contains the partitions $\mathcal P$ that can be formed in the following way:

First, take a hyperplane $H$ orthogonal to $v$ (possibly at $+\infty$ or $-\infty$) and split $\rr^d$ into two sets using $H$.  Let $H^+ = \{ h + \alpha v : h \in H, \alpha \ge 0\}$ and $H^- = \{ h - \alpha v : h \in H, \alpha \ge 0\}$.  Let $j$ be the non-negative integer such that $Y_1 \in V^j_d$ and $Y_2 \in V^{n-j-1}_d$.  Then, take two partitions $\mathcal{P}_1 \in \C^j_{Y_1}$ and $\mathcal{P}_2 \in \C^{n-j-1}_{Y_2}$.  The sets $K_1 \cap H^+$ and $K_2 \cap H^-$ where $K_1$ ranges over all sets in $\mathcal{P}_1$ and $K_2$ ranges over all sets in $\mathcal{P}_2$ form the partition $\mathcal P$.
Since $\mathcal P$ is determined by $H$, $\mathcal{P}_1$ and $\mathcal{P}_2$, we may describe it by the triple $(H, \mathcal{P}_1, \mathcal{P}_2)$.

It should be stressed that, if $H$ is not at $\pm \infty$, we distinguish between $(H, \mathcal{P}_1, \mathcal{P}_2)$ and $(H, \mathcal{P}'_1, \mathcal{P}'_2)$ as long as $( \mathcal{P}_1, \mathcal{P}_2) \neq ( \mathcal{P}'_1, \mathcal{P}'_2)$, even if the actual partitions are the same.  If $H$ is at $+\infty$ (i.e. $H^+$ is empty) we identify $(H, \mathcal{P}_1, \mathcal{P}_2)$ with $(H, \mathcal{P}'_1, \mathcal{P}_2)$ for all $\mathcal{P}_1, \mathcal{P}'_1 \in \C^j_{Y_1}$.  Analogously, if $H$ is at $-\infty$ we identify $(H, \mathcal{P}_1, \mathcal{P}_2)$ with $(H, \mathcal{P}_1, \mathcal{P}'_2)$ for all $\mathcal{P}_2, \mathcal{P}'_2 \in \C^{n-j-1}_{Y_2}$.  This is so that we can recursively equip $\C^n_Y$ with a topology that makes it homeomorphic to $\C^j_{Y_1} * \C^{n-j-1}_{Y_2}$.

\begin{figure}[htc]
\centering
\includegraphics{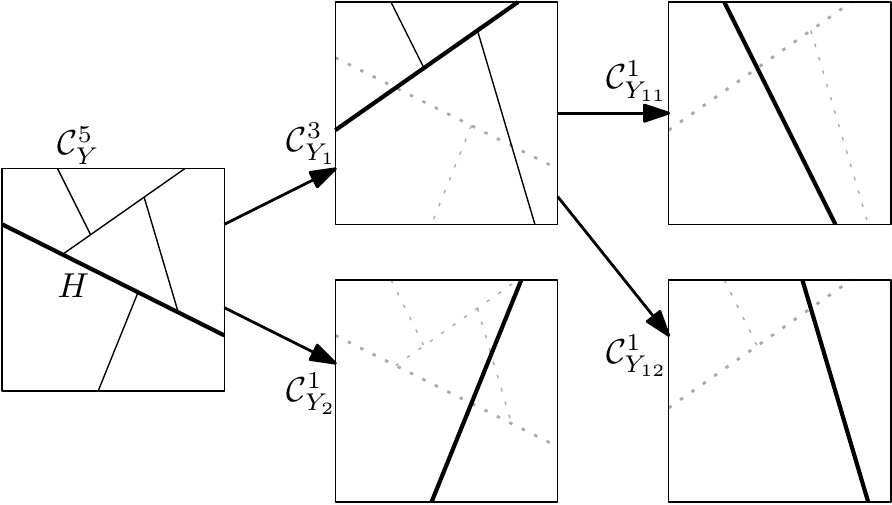}
\caption{Example of a partition induced by some $Y \in V^5_2$ and its recursive definition.  If these lines are translated independently, as long as they cut the same regions, they still give partitions in $\mathcal C^5_Y$.}
\label{fig-V5}
\end{figure}

For $Y \in V^n_d$, all partitions of $\C^n_Y$ split $\rr^d$ into $n+1$ convex sets, some possibly empty (see Figure \ref{fig-V5}).  Moreover, they can be made by successively cutting one of the remaining parts by a hyperplane, using $n$ hyperplanes in total.  For instance, if all the directions involved in $Y$ are parallel, we obtain a partition by parallel hyperplanes.

Given a positive integer $k$, we label each of the $n+1$ parts of a partition of $\C^n_Y$ with an element of $[k] = \{1,2,\ldots, k\}$.  This can be thought of as distributing the parts of the partition among $k$ thieves. We denote by $\C^n_{Y,k}$ the space of pairs $(\mathcal P,\ell)$ where $\mathcal P$ is a partition in $\C^n_Y$ and $\ell:\mathcal P\to [k]$ is such a labelling. Once again, when $H$ is at $+\infty$, we identify $(H, \mathcal{P}_1, \mathcal{P}_2)$ with $(H, \mathcal{P}'_1, \mathcal{P}_2')$ whenever $\mathcal{P}_2=\mathcal{P}_2'$ and the labelling of the sets contained in $H^-$ coincides for both partitions.
The analogous identifications are also made when $H$ is at $-\infty$.

Even though different elements of $V^n_d$ may give wildly different types of partitions, the topology of $\C^n_{Y,k}$ for $Y \in V^n_d$ is surprisingly easy to describe.  For this, we use a set $G$ of $k$ elements with the discrete topology.

\begin{lemma}
For a non-negative integer $n$, and $Y \in V^n_d$, the space $\C^{n}_{Y,k}$ is homeomorphic to $EG^n$, the $(n+1)$-fold join of $G$.
\end{lemma}

\begin{proof}
As before, it suffices to note that for $n=0$, the result holds.  For positive $n$ and $Y = (v, Y_1, Y_2)$, let $k$ be the integer such that $Y_1 \in V^j_d$ and $Y_2 \in V^{n-1-j}$.  Noticing that $\C^n_{Y,k} \cong \C^j_{Y_1,k} * \C^{n-j-1}_{Y_2,k}$ and that the topological join is associative completes the proof by induction.
\end{proof}

Moreover, we can consider the $\Sigma_k$ action in $\C^n_{Y,k}$, where $\sigma \in \Sigma_k$ acts by
$\sigma(\mathcal P,\ell)=(\mathcal P,\sigma^{-1} \circ \ell)$ (i.e. it permutes the distribution of the sets according to $\sigma$).  Then, the homeomorphism mentioned above is equivariant with the usual action of $\Sigma_k$ in $EG^n$.

This homeomorphism readily implies a generalisation of the necklace splitting problem.

\begin{theorem}\label{teorema-dimension-grande}
Let $k$ be a prime power, $t$ a positive integer and $n= (k-1)t$.  Given $t$ probability measures $\mu_1, \mu_2, \ldots, \mu_t$ in $\rr^d$, for any $Y \in V^n_d$, there is a partition in $\C^n_{Y,k}$ of $\rr^d$ into $k$ parts $A_1, A_2, \ldots, A_k$ such that
\[
\mu_j (A_i) = \frac{1}{k}
\]
for each $i,j$.
\end{theorem}

\begin{proof}
For $1 \le j \le t$, consider the function
\begin{align*}
f_j : \C^n_{Y,k} &\to \rr^{k-1} \\
(\mathcal P,\ell) &\mapsto \left( \mu_j (A_1) - \frac{1}{k}, \mu_j (A_2) - \frac{1}{k}, \dots, \mu_j (A_k) - \frac{1}{k} \right),
\end{align*}
where $A_i$ is the union of the sets with label $i$.  Even though the image of $f_j$ has $k$ coordinates, their sum is always $0$, so the dimension can be considered to be $k-1$.  Notice that $\rr^{k-1}$ has an action of $\Sigma_k$ which simply permutes the $k$ coordinates.  Moreover, each $f_j$ is equivariant with respect to the $\Sigma_k$ action in both spaces.

Thus, the function
\begin{align*}
f: \C^n_{Y,k} &\to  \rr^{t(k-1)} \\
f &= (f_1, f_2, \dots, f_t)
\end{align*}
Is an equivariant function between these two spaces.  If $0$ is in the image of $f$, we have a partition satisfying our condition.  If not, then we have an equivariant map $f: EG^{t(k-1)} \rightarrow \rr^{t(k-1)}\setminus \{ 0 \}$, contradictin \"Ozaydin's theorem.
\end{proof}

If $k$ is prime, then one can complete the proof above using Dold's theorem instead of the stronger result by \"Ozaydin.  For this, consider only the action of $\Z_k$ that simply permutes the distribution cyclically, instead of the whole action of $\Sigma_k$.  If $0$ is not in the image of $f$, we can consider $g = \frac{f}{|| f ||}$, which is an equivariant function from $\C^n_Y$ to $S^{t(k-1)-1}$.  Since $k$ is a prime number, the $\Z_k$-action is free on both spaces.
However, since $\C^n_Y \cong EG^{t(k-1)}$ is $[t(k-1)-1]$-connected and $S^{t(k-1)-1}$ is of dimension $t(k-1)-1$, we would be contradicting Dold's theorem.

Notice that the value of $n$ is optimal. To see this, place each measure $\mu_i$ concentrated at a point $x_i$ so that no vector in the definition of $Y$ is orthogonal to any vector of the form $x_i - x_j$. Then each $x_i$ must be contained in at least $k-1$ hyperplanes.  In the case $(d,k) = (1,2)$ our proof is essentially the same as the one in \cite{Alon:1985cy}.

It would be desirable to remove the condition of $k$ being a prime power in the theorem above, but usual subdivision arguments fail to work in this case.  The typical subdivision argument seem only to work with $d=1$, as the product of two partitions in $\rr$ into $m$ and $n$ parts is a partition of $\rr$ into $m+n-1$ parts.  Moreover, if $k$ is not a prime power, there are examples of equivariant maps between the spaces mentioned at the end of the proof \cite{Oza87, Volovikov:1996up}, so the argument above cannot be extended by trying to improve the topological result.  A paper by de Longueville explains how similar methods are used to tackle some cases of the topological Tverberg theorem \cite{DeLongueville:2001vw}.  It should be noted that even though the topological tools fail in these cases, the mass partition result may still hold for every value of $k$.

\begin{question}
Does Theorem \ref{teorema-dimension-grande} hold for all $k$?
\end{question}

Even with the conditions on $k$, Theorem \ref{teorema-dimension-grande} readily implies a generalisation of the main result in \cite{deLongueville:2006uo}, namely

\begin{theorem}\label{teorema-collaresgrandes}
Let $k, t, d$ be positive integers, and $n=t(k-1)$.  Given $t$ probability measures $\mu_1,\mu_2\dots,\mu_t$ in $\rr^d$ and $n$ directions $v_1, v_2, \ldots, v_n$, there is a partition of $\rr^d$ induced by hyperplanes $H_1,H_2,\dots,H_n$, each orthogonal to the corresponding $v_i$, such that its parts can be distributed among $k$ sets and each obtains $\frac{1}{k}$ of each measure.
\end{theorem}

\begin{proof}
Notice that Theorem \ref{teorema-dimension-grande} implies Theorem \ref{teorema-collaresgrandes} if $k$ is a prime number.  Thus, it suffices to show that if the result holds for $k=l$ and $k=m$, then it also holds for $k=ml$.

Given $k=ml$ and $t$ measures, consider a partition $\mathcal{P}$ by $t(m-1)$ hyperplanes distributed into $m$ sets $B_1, B_2, \ldots, B_m$ such that $\mu_j (B_i) = 1/m$ for each $i,j$.

Then, for a fixed $1 \le i \le m$, consider the measures $\mu^i_j = \mu_j |_{B_i}$ the restriction to $B_i$.  We can consider a partition $\mathcal{P}_i$ of $B_i$ using $t(l-1)$ hyperplanes and a distribution of the parts into $l$ sets so that each of the measures $\mu^i_j$ for $1 \le j \le t$ is split evenly among the $l$ sets.  Notice that $\cup_{i=1}^k \mathcal{P}_i$ is a partition of $\rr^d$ distributed into $ml$ parts which splits evenly each measure $\mu_j$.  Moreover, the total number of hyperplane cuts used to make it is $t(m-1) + m\cdot t(l-1) = t(ml-1)$, as desired.
\end{proof}

In \cite{deLongueville:2006uo}, the directions of the vectors $v_i$ could only be chosen among $d$ possibilities, though here we can take an arbitrary number of them.  The downside with respect to Theorem \ref{teorema-dimension-grande} is that we do not have a good control over the number of pieces obtained at the end.
\begin{question}
What is the minimum nuber $n=n(t,d,k)$ such that for any $t$ probability measures in $\rr^d$, there is a partition of $\rr^d$ into $n$ convex pieces so that its parts can be distributed into $k$ sets $A_1, A_2, \ldots, A_k$ satisfying
\[
\mu_j (A_i) = \frac1k
\]
for all $i, j$.
\end{question}

If we ask for partitions induced by successive hyperplane cuts with fixed directions of the remaining parts, then for $k$ a prime power Theorem \ref{teorema-dimension-grande} shows that $n = t(k-1)+1$ is optimal.  However, relaxing the conditions on the partitions may allow us to find smaller values of $n$.  For instance, if $d \ge t$, then $n = k$ is sufficient.  This comes from the fact that for any $d$ measures in $\rr^d$, there is a partition of $\rr^d$ into $k$ convex parts so that each part has the same size in each measure.  Different proofs of this result and its generalisations are found in \cite{Soberon:2012kp, Karasev:2014gi, Blagojevic:2014ey}

As expected, Theorem \ref{teorema-dimension-grande} implies the necklace splitting theorem by Alon.  The proof is essentially different.  In \cite{Alon:1987ur}, the topological tool used is one that guarantees the existence of a certain equivariant map, and then this map is used to obtain the partition.  In our setting, we use a topological tool that guarantees the non-existence of a certain map.  Then, if the result was false we obtain a contradiction to this claim.

\section{Unified approach to fair distribution problems}\label{section-karasev}

In this section, we aim to give another fair distribution result which contains all known high-dimensional generalisations of the necklace splitting theorem, including the ones in the previous section.
The partitions we use are generalisations of Voronoi diagrams, similar to the ones in \cite{Karasev:2014gi}.

We consider 
\begin{itemize}
\item 
A topological space $X$ (usually a manifold);
\item
A set of $t$ probabilistic Borel measures $\mu_1, \ldots, \mu_t$ on $X$;
\item 
And $n=t(k-1)+1$ continuous functions $f_1,\ldots, f_n : X\to\mathbb R$.
\end{itemize}

We also need the following technical assumptions: for any constant $c$ and any $i,j\in[1,n]$ the set $\{f_i - f_j = c\}$ has measure zero with respect to any $\mu_j$.

Now we consider a weight-vector $\bar c=(c_1,\dots,c_n)$ and define a partition of $X$ into $n$ sets of the form
$$
V_i(\bar c) = \{x\in X : \forall j\neq i\ f_i(x) + c_i \ge f_j(x) + c_j\}.
$$
If $X=\mathbb R^n$ and the functions $f_i$ are linear then this is just a generalized Voronoi partition into convex parts, also called a regular partition, or affine partition in~\cite{Anonymous:FbJ-jtiN}. Fixing the functions $f_i$ and only varying the constants $c_i$ corresponds to translating the walls between $V_i(\bar c)$ and $V_j(\bar c)$. Regular partitions into convex parts also exist when $X$ is a sphere or the hyperbolic space (see~\cite{Karasev:2014gi}). 

When the space $X$ is not compact, we allow some, but not all, of the $c_i$ to be $-\infty$. This is a sort of compactification that is necessary for the result below, it allows us to make some of the $V_i(\bar c)$ empty.
 
Now we may state the generalised splitting necklace theorem:

\begin{theorem}
\label{theorem:splitting}
Let $k$ be a prime power and $X$, $\{\mu_j\}_{\ell=1}^t$ , $\{f_i\}_{i=1}^n$, $n=t(k-1)+1$ as above.  Then there is a weight-vector $\bar c$ such that the elements $V_i(\bar c)$ of the induced partition can be distributed among $k$ sets $A_1,\ldots, A_k$ with
\[
\mu_j(A_i) = 1/k
\]
for all $i=1,\dots,k$ and $j=1,\dots, t$.
\end{theorem}

Let us show how one can infer the other results from this. Noga Alon's theorem follows by considering $X=\rr$, and taking $f_i$ to be linear functions with different slopes.  This gives us a partition of $\rr$ using at most $n-1$ points which can be distributed evenly among $k$ people.  To remove the condition of $k$ being a prime power, a subdivision argument like the one described in Section \ref{section-neckalces} is sufficient.

The result of~\cite{Zivaljevic:2013vla} is obtained by taking $X$ to be the interior of a simplex of dimension $n-1$, $f_i$ to be the logarithm of the distance from $x\in X$ to a hyperplane of the $i$-th facet of $X$.  If $c_i, c_j$ are not both equal to $-\infty$, the boundary between $V_i(\bar c)$ and $V_j (\bar c)$ is the set of points $x$ such that 
\[
\mbox{dist}(x,F_j) = \alpha_{i,j} \cdot \mbox{dist}(x,F_j)
 \]
 where $F_i, F_j$ are the $i$-th and $j$-th facets of the simplex, and $\alpha_{i,j}$ is a non-negative number depending only on $c_i, c_j$.  If $c_i = c_j = -\infty$, $V_i (\bar c)$ and $V_j(\bar c)$ are both empty.  Then, any partition $\{V_i(\bar c)\}$ is just a conical partition of the simplex with the position of the apex depending on $\bar c$.  More precisely,
 
 \begin{corollary}
Suppose $k$ is a prime power, $n= t(k-1)+1$ and $\Delta$ is an $(n-1)$-dimensional simplex.  For any $t$ probability measures $\mu_1, \mu_2, \ldots, \mu_t$ that vanish on every hyperplane, there is a point $x \in \Delta$ such that the conical partition of $\Delta$ induced by the facets of $\Delta$ and common apex $x$ can be distributed among $k$ sets $A_1, A_2, \ldots, A_k$ such that for all $i,j$
\[\mu_j (A_i) = 1/k.
\]
 \end{corollary}

The result of Section \ref{section-neckalces} needs more explanation.  The space of nested partitions by hyperplanes with fixed directions cannot be obtained directly by the construction above. However,~\cite[Lemma~3]{Anonymous:FbJ-jtiN} shows that any iterated regular partition is a limit of (non-iterated) regular partitions. So the result for nested partitions into $n=t(k-1)+1$ parts with prescribed directions of hyperplanes also follows by an appropriate limiting argument.

\begin{proof}[Proof of Theorem~\ref{theorem:splitting}]

First, we consider the functions $f_i$ as a map $f : X\to \mathbb R^n$. This map is continuous and therefore Borel, so it pushes the Borel measures $\mu_j$ forward to the space $\mathbb R^n$. In the rest of the proof we restrict ourselves to the following special case: $X=\mathbb R^n$, $f_i$ are its coordinate functions, and $\mu_j$ are Borel probability measures that assign zero to any hyperplane defined by an equation of the form $f_i-f_j=c$.

Let us take a different parameter instead of $\bar c\in \mathbb \rr^n$. Choose a probability measure $\mu$ on $X$ whose support is the whole space $X$, which is also zero on every set of the form $\{f_i - f_j = c\}$ and such that $\mu(A) \neq 0$ for every non-empty open set $A \subset X$. For each $i$, consider
\[
w_i = \mu(V_i(\bar c)). 
\]
It is known (see~\cite{Aurenhammer:1998tj} and \cite{Karasev:2014gi}) that these $w_i$'s continuously parameterise the space of all partitions of the form $\{V_i(\bar c)\}$, which we now denote by $\{V_i(\bar w)\}$.  The non-vanishing condition on the measure $\mu$ is essential for this part. All these coordinates are nonnegative and their sum is $1$. The possibility of having $c_i=-\infty$ in the case when $X$ is not compact also ensures that any collection (but not all) of $w_i$'s can be made zero. So the space of such partitions becomes a simplex $\Delta^{n-1}$. 

If we augment the coordinates $(w_i)$ with the assignment of a number $n_i\in[k]$ to each of the sets $\{V_i(\bar w)\}$, then we obtain the first approximation to the configurations space of our problem, which is the union of $k^n$ disjoint copies of the $(n-1)$-dimensional simplex $\Delta^{n-1}$. We then take the quotient of this space by the following relation: If $w_i=0$ (that is, $V_i(\bar w)$ is empty) then we ignore the assignment of $n_i$ to this set.

Now the simplices glue together and give the $n$-fold join $[k]*\dots *[k]$ of $k$-element sets. As in Section \ref{section-neckalces}, this space is simply $EG^{n-1}$ together with the natural action of $\Sigma_k$.

The sums 
\[
\phi_{\ell, j} =\sum_{n_i = \ell} \mu_j(V_i(\bar w)) - 1/k
\]
arrange into a $\Sigma_k$-equivariant map $\phi : EG^{n-1} \to \rr^{n-1}$, exactly as in the proof of Theorem \ref{teorema-dimension-grande}.  Thus, we know that if $k$ is a prime power, the map $\phi$ must have a zero.  This is equivalent to the existence of the certain partition $\{V_i(\bar w)\}$ and the assignment $\{n_i\}$ giving the conclusion of the theorem.
\end{proof}

In \cite{Karasev:2014gi}, similar partitions are used. More freedom is given to the set of functions, which allows for generalisations of the ham-sandwich theorem where the directions are not fixed.

\section{High-dimensional chessboards}\label{seccion-ajedrez}

One of the disadvantages of the theorems in the previous sections is that, after splitting $\rr^d$, we have no control over how the distribution is made among the $k$ thieves.  In \cite{Alon:1985cy}, a result with more control is shown in the cube $[0,1]^d$, where the distribution is made by a chessboard colouring.

To be precise, given $r+1$ numbers $0=z_0< z_1 < z_2 < \ldots < z_{r-1} < z_r =1$, they define the intervals $I_i = [z_i, z_{i+1}]$ for $0\le i \le r-1$.  They call the partition of $[0,1]^d$ given by sets of the form $I_{n_1} \times I_{n_2} \times \cdots \times I_{n_d}$ a partition of size $r-1$.  They show that for any $k$ measures in $[0,1]^d$, a chessboard colouring of a partition of size at most $k$ can split them evenly at the same time if $d$ is odd.

However, in this result one has the condition that the cube is being split the same way in each direction.  We would like to know the complexity of chessboard colourings where the inducing hyperplanes have more freedom.

Given an ordered $t$-tuple of positive integers $n=(n_1, n_2, \ldots, n_t)$ and a $t$-tuple $v= (v_1, v_2, \ldots, v_t)$ of directions in $\rr^d$, we say that a partition $(A,B)$ of $\rr^d$ is an $(n,v)$-chessboard colouring if it can be obtained by taking at most $n_i$ hyperplanes orthogonal to $v_i$ for each $i$ and using the resulting family to give a chessboard colouring of $\rr^d$ (i.e. any two regions that share a boundary of dimension $d-1$ are in different sets).

\begin{figure}[htc]
\centering
\includegraphics[]{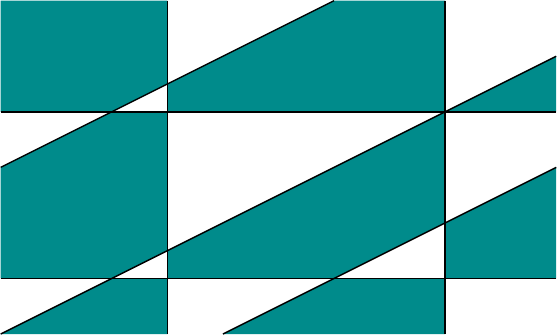}
\caption{Example of a $2$-dimensional chessboard using the vector $(3,2,2)$}
\label{fig-ajedrez}
\end{figure}

Let $S(n)$ be the sum of the entries of $n$.
\begin{question}
Determine the set $\mathcal{N}$ of $t$-tuples $n$ such that for every $t$-tuple $v$ of directions in $\rr^d$ and every $S(n)$ measures in $\rr^d$ there is an $(n,v)$-chessboard colouring $(A,B)$ of $\rr^d$ which splits each measure in half.
\end{question}

Note that $S(n)$ is the maximum number of measure we could hope to be able to split.  It is somewhat surprising that $\mathcal{N}$ is not the set of all $t$-tuples of positive integers.  In the next section we show that $(1,1) \not\in \mathcal{N}$.

\begin{theorem}\label{teorema-ajedrez}
Let $n= (n_1, n_2, \ldots, n_t)$.  If the multinomial coefficient
\[
{S(n)}\choose{n_1, n_2, \ldots, n_k}
\]
is odd, then $n \in \mathcal{N}$.
\end{theorem}

The condition above holds if and only if $n_i$ and $n_j$ do not share a $1$ in the same position in their binary expansions.  Moreover, for $n=(n_1, n_2, \ldots, n_t) \not\in \mathcal{N}$, the result above also gives a lower bound for the maximum number $L$ of measures that can be split by some $(n,v)$-chessboard colouring for a prescribed $v$.  Namely, it suffices to find integers $m_1 \le n_1, m_2 \le n_2, \ldots, m_t \le n_t$ such that $m=(m_1, m_2, \ldots, m_t) \in \mathcal{N}$.  Then, $S(m)$ is a lower bound for the number of measures we can split with the $t$-tuple $n$.  We would like to thank Albert Haase, Florian Frick and Pavle Blagojevi\'c for pointing out these facts.

\begin{proof}[Proof of Theorem \ref{teorema-ajedrez}]
For each $i$, let $Y_i \in V^{n_i}_d$ be a set formed by only taking the direction $v_i$ in all instances of the recursive definition of $Y_i$.  Notice that $\C^{n_i}_{Y_i,2}$ is the set of chessboards colourings $(A_i, B_i)$ using at most $n_i$ hyperplanes all orthogonal to $v_i$ (if the colour of two consecutive sections is the same, we can ignore the hyperplane dividing them).

Now consider the space $\C_0 = \C^{n_1}_{Y_1,2} \times \C^{n_2}_{Y_2,2} \times \ldots \times \C^{n_t}_{Y_t,2} \cong S^{n_1} \times S^{n_2} \times \cdots \times S^{n_t}$.  Every element of this set induces an $(n,v)$-chessboard colouring $(A,B)$ by simply considering
\begin{align*}
A & = \{ x \in \rr^d : x \in A_i \ \mbox{for an odd number of }i\} \\
B & = \{ x \in \rr^d : x \in A_i \ \mbox{for an even number of }i\}
\end{align*}

Moreover, notice that flipping any particular $(A_i, B_i)$ by $(B_i, A_i)$ automatically flips $(A,B)$ as well.

If $\mu_1, \mu_2, \ldots, \mu_{S(n)}$ are probability measures, consider the function
\begin{align*}
f: \C_0 &\rightarrow  \rr^{S(n)} \\
(A,B) &\mapsto \left( \mu_1 (A) - \frac12, \mu_2(A) - \frac12, \ldots, \mu_{S(n)}(A)-\frac12 \right)
\end{align*}

Notice that $f(A,B) = -f(B,A)$.  Thus we have a function from a products of spheres to $\rr^d$ which is equivariant with respect to the action of $\Z_2$ on each of the spheres of the product.  The conditions that guarantee a zero of functions of this kind have been studied by Fadell and Husseini \cite[Example~3.3 and Section 5]{fadell1988ideal} and later made explicit by Ramos \cite[Theorem~3.1]{Ramos:1996dm}.  Ramos' result applied to our setting implies that the function above has a zero if the multinomial coefficient ${S(n)}\choose{n_1, n_2, \ldots, n_k}$ is odd, and a zero of this function is precisely what we want.
\end{proof}

For an algorithm that computes the index described in \cite{fadell1988ideal} that works for our setting, we recommend \cite[Section 3.4]{blagojevic2011ideal}.  In \cite{Ramos:1996dm}, Theorem 3.1 describes the existence of zeroes of a much larger class of functions.  Namely, he considers functions where the action on each sphere may be equivariant or stable (i.e. $f(x) = f(-x)$) on each coordinate of the image independently, prescribed in advance by a $0,1$ matrix.  The guarantee of a zero in these functions is given in terms of the parity of a certain permanent.

The full power of Ramos' result can be used in this setting, where each measure $\mu_j$ is to be split using only the hyperplanes in some of the directions (possibly not all), prescribed in advance.  For example, given our set of hyperplanes $\mu_1$ might be split using only the hyperplanes in directions $v_1$ and $v_3$, while $\mu_2$ might be split using only the hyperplanes in directions $v_2$ and $v_3$.

The proof method to guarantee that simultaneously each measure is split by half by its corresponding partition is essentially the same as the one presented above, so we refer the reader to \cite{Ramos:1996dm} for the conditions needed on $n$.  However, this result seems quite artificial compared to the partitions in Theorem \ref{teorema-ajedrez}.

It should be clear that if $(n_1, n_2, \ldots, n_t) \in \mathcal{N}$, then the $(t-1)$-tuple $(n_1 + n_2, n_3, \ldots, n_t) \in \mathcal{N}$.  However, using this last observation does not improve Theorem \ref{teorema-ajedrez}.

\section{Negative results}\label{section-negative}

In this section, we exhibit two examples showing that sometimes we cannot split as many measures as we would hope from the number of degrees of freedom involved.

\begin{claim}
The $2$-tuple $(1,1)$ is not in $\mathcal{N}$.
\end{claim}

\begin{proof}
It suffices to exhibit two measures in $\rr^2$ which cannot be split evenly at the same time by a chessboard colouring using a horizontal line and a vertical one. The two lines are determined by their intersection point $(x,y)$.

Consider the measure $\mu_1$ to be concentrated uniformly in a segment with positive slope.
Then the locus $Z_1$ of points $(x,y)$ at the intersection of the two lines such that the induced chessboard colouring splits $\mu_1$ evenly is easy to describe.  It is composed of two paths, each made by $3$ straight segments (see Figure \ref{fig-(1,1)} (a)).

To define $\mu_2$, simply translate $\mu_1$ slightly in the direction orthogonal to the segment defining $\mu_1$. Then $Z_1$ and the analogous locus $Z_2$ corresponding to $\mu_2$ do not intersect (see Figure \ref{fig-(1,1)} (b)).

\begin{figure}[htc]
\centering
\includegraphics{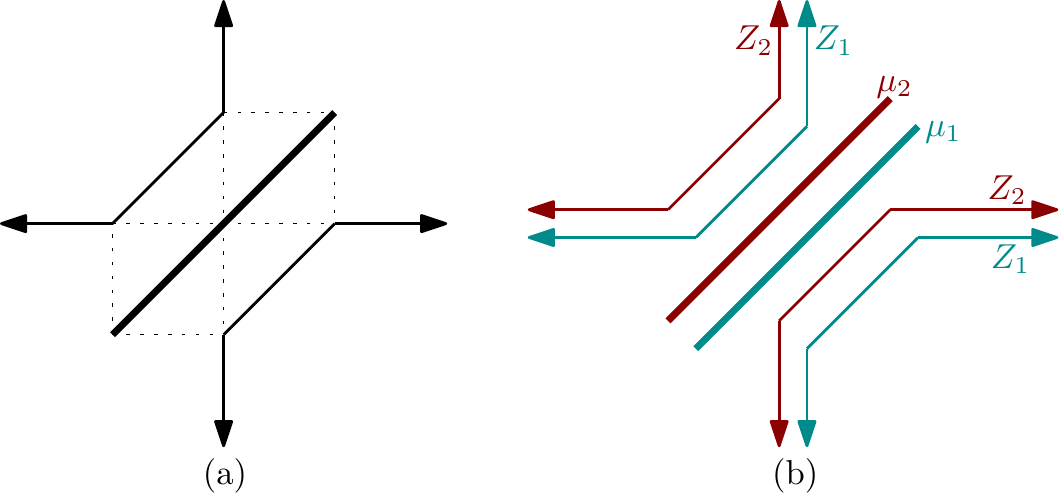}
\caption{On the left, the dotted lines go through the midpoint or endpoints of the segment describing $\mu_1$. On the right, after translating, the loci $Z_1$ and $Z_2$ do not intersect.}
\label{fig-(1,1)}
\end{figure}

Sufficiently close approximations of $\mu_1$ and $\mu_2$ implies the result for absolutely continuous measures.
\end{proof}

Our second example is regarding a high-dimensional extension of the results of Section \ref{section-paths}.  Even though the polynomial ham sandwich works in higher dimensions, the proper way to emulate algebraic surfaces of a given degree by surfaces made of hyperplane parts of fixed directions is unclear.

In the plane, the case with $1$ turn basically says that for any $2$ measures there is the translation of a quadrant that contains exactly half of each measure.  Divisions by a vertical or horizontal line can be thought as subcases where the vertex of the quadrant is at infinity.

Thus, one can ask the same question for $\rr^d$, since orthants have $d$ degrees of freedom.  In other words, for any $d$ measures one might expect to find an orthant which contains exactly half of each.  It is surprising that this is not true.

\begin{claim}
For any fixed $d$, there are $3$ measures such that no orthant contains exactly half of each.
\end{claim}

\begin{proof}
Consider any direction $v$ which is not parallel to any vector of the standard basis of $\rr^d$.  Take measures $\mu_1$, $\mu_2$, $\mu_3$ concentrated near the points $v, 2v, 3v$, respectively.  Notice that no orthant can have these $3$ points in its boundary.  Thus, a sufficiently close approximation of these measures by absolutely continuous measures gives the desired construction.
\end{proof}

However, the following question remains interesting.
\begin{question}
Is it possible to find a ``fixed directions'' analogue of the polynomial ham sandwich theorem in high dimensions?
\end{question}

\section{Acknowledgments}

The authors would like to thank Ruy Fabila-Monroy for bringing the question of paths with few turns to our attention, and to Pavle Blagojevi\'c, Florian Frick and Albert Haase for the stimulating discussions and clarifications regarding the topological tools needed for this work.

\bibliographystyle{amsalpha}

\bibliography{references}

\noindent R. Karasev \\
\textsc{
Moscow Institute of Physics and Technology \\
Institutskiy per. 9, Dolgoprudny \\
Russia 141700}\\[0.3cm]
\noindent E. Rold\'an-Pensado \\
\textsc{
Instituto de Matem\'aticass, Unidad Juriquilla \\
Universidad Nacional Aut\'onoma de M\'exico \\
Juriquilla Quer\'etaro 76230, M\'exico}\\[0.3cm]
\noindent P. Sober\'on \\
\textsc{
Mathematics Department \\
University of Michigan \\
Ann Arbor, MI 48109-1043
}\\[0.3cm]

\noindent \textit{E-mail addresses: }\texttt{r\_n\_karasev@mail.ru, e.roldan@im.unam.mx, psoberon@umich.edu}

\end{document}